\documentclass[12pt,leqno]{amsart}
\newtheorem{thm}{Theorem}[section]
\newtheorem{defi}{Definition}[section]
\newtheorem{lem}{Lemma}[section]
\newtheorem{cor}{Corollary}[section]

\newtheorem{rem}{Remark}[section]

\newtheorem{ex}{Example}[section]
\newcommand{\be}{\begin{equation}}
\newcommand{\ee}{\end{equation}}
\newcommand{\bea}{\begin{eqnarray}}
\newcommand{\eea}{\end{eqnarray}}
\newcommand{\beb}{\begin{eqnarray*}}
\newcommand{\eeb}{\end{eqnarray*}}
\usepackage{amssymb,amsfonts,amsthm,setspace,indentfirst}
\numberwithin{equation}{section}
\begin{document}
%%%%%%%%%%%%%%%%
%
\title[Curvature properties of Melvin magnetic metric]{Curvature properties of Melvin magnetic metric}
\author[A. A. Shaikh, Akram Ali, Ali H. Alkhaldi and D. Chakraborty]{Absos Ali Shaikh$^1$, Akram Ali$^2$, Ali H. Alkhaldi$^2$ and Dhyanesh Chakraborty$^1$}
\date{}
\address{\noindent\newline$^1$ Department of Mathematics,
\newline University of Burdwan, 
\newline Golapbag, Burdwan-713104,
\newline West Bengal, India} 
\email{aask2003@yahoo.co.in, aashaikh@mathburuniv.ac.in}
\email{dhyanesh2011@gmail.com}

\address{\noindent\newline$^2$ Department of Mathematics,
	\newline College of Science,
	\newline King Khalid University, 
	\newline 9004 Abha, Saudi Arabia} 
\email{akramali133@gmail.com ; akali@kku.edu.sa}
\email{ahalkhaldi@kku.edu.sa}
%
%%%%%%%%%%%%%%%%%%%%%%%%%%%%%%%%%%%%%%%%%%%%%%%%%%%%%%%%%%%%%%%%%%%%%%%%%%%%%%%%%%%%%%%%%%%%%%%%
%
%
%%%%%%%%%%%%%%%%%%%%%
\begin{abstract}
This paper aims to investigate the curvature restricted geometric properties admitted by Melvin magnetic spacetime metric, a warped product metric with $1$-dimensional fibre. For this, we have considered a Melvin type static, cylindrically symmetric spacetime metric in Weyl form and it is found that such metric, in general, is generalized Roter type, $Ein(3)$ and has pseudosymmetric Weyl conformal tensor satisfying the pseudosymmetric type condition $R\cdot R-Q(S,R)=\mathcal L' Q(g,C)$. The condition for which it satisfies the Roter type condition has been obtained. It is interesting to note that Melvin magnetic metric is pseudosymmetric and pseudosymmetric due to conformal tensor. Moreover such metric is $2$-quasi-Einstien, its Ricci tensor is Reimann compatible and Weyl conformal $2$-forms are recurrent. The Maxwell tensor is also pseudosymmetric type.
\end{abstract}
%%%%%%%%%%%%%%%%%%%%%
\noindent\footnotetext{ $^*$ Corresponding author.\\
$\mathbf{2010}$ \hspace{5pt} Mathematics\; Subject\; Classification: 53B20, 53B25, 53B30, 53B50, 53C15, 53C25, 53C35.}
\maketitle
%
%%%%%%%%%%%%%%%%%%%%%%%%%%%%%%%%%%%%%%%%%%%%%%%%%%%%%%%%%%%%%%%%%%%%%%%%%%%%%%%%%%%%%%%%%%%%%%%%%%%%%%%%%%%%%%%%%%%%%%%%%%%%%%%%%%%%%%%%%%%%%%%%%%%%%%%%
%																																Introduction
%%%%%%%%%%%%%%%%%%%%%%%%%%%%%%%%%%%%%%%%%%%%%%%%%%%%%%%%%%%%%%%%%%%%%%%%%%%%%%%%%%%%%%%%%%%%%%%%%%%%%%%%%%%%%%%%%%%%%%%%%%%%%%%%%%%%%%%%%%%%%%%%%%%%%%%%
\section{\bf{Introduction}}
%==========================================
%
\indent We consider $M$ as a connected and smooth manifold of dimension $n(\ge 3)$, on which a semi-Reimannian metric $g$ with signature $(\sigma,n-\sigma)$, $0\le \sigma \le n$, is endowed. Then $(M,g)$ is called an $n$-dimensional Reimannian (resp, Lorentzian) manifold if $\sigma =0$ or $n$ (resp, $\sigma =1$ or $n-1$). The Levi Civita connection, the Riemann curvature tensor, the Ricci tensor and the scalar curvature of $M$ are respectively denoted by $\nabla$, $R$, $S$ and $\kappa$. We mention that a spacetime is a $4$-dimensional connected Lorentzian manifold.\\
\indent A stationary, cylindrically symmetric electrovac solution to Einstein-Maxwell system of equations in general relativity was obtained by Melvin \cite{Melvin64} in $1964$ and such solution is often referred as  `Melvin Magnetic Universe'. Physically, this spacetime describes a bundle of parallel magnetic lines of force which under their mutual gravitational attraction remain in equilibrium. It is interesting to note that such spacetime is not assymptotically flat and has not singularity but is geodesically complete. About an axis of symmetry the spacetime is invariant under rotation and translation, and also invariant under reflection in planes containing that axis or perpendicular to it. Moreover it admits four Killing vectors and Weyl conformal tensor is Petrov type D.\\
\indent Melvin himself in \cite{Melvin65}  proved his spacetime to be stable against small radial perturbations. Thorne \cite{Thorne65} further showed the stability of such spacetime against any large perturbations and also suggested that this spacetime has a great value in the study of extragalactic sources of radio waves. In this paper we are really tempted to investigate the geometric structures of such magnetic spacetime. \\
\indent The geometry of Melvin magnetic spacetime in cylindrical symmetry is described by the metric
\bea\label{MMM}
ds^2=U_B^2(-dt^2 +dr^2 +dz^2) +\frac{r^2}{U_B^2}d\phi^2
\eea
where $U_B=1 +\frac{B_0r^2}{4}$, $B_0$ is the magnetic field at $r=0$ axis. The Maxwell field is given by
\bea\label{MF}
F=\frac{B_0r^2}{U_B^2}d\phi \wedge dr^2.
\eea
The coordinates $t$, $r$, $z$ and $\phi$ are respectively dimensionless time, radial, longitudinal and azimuthal angle. It is easy to see that for vanishing magnetic field, $U_B=1$ and the metric (\ref{MMM}) reduces to the Minkowski metric in cylindrical coordinates. It may be noted that the metric \eqref{MMM} can be written as a warped product metric 
$$ds^2= d\bar{s}^2 + F^2d\tilde{s}^2$$ where $d\bar{s}^2=(1 +\frac{B_0r^2}{4})^2(-dt^2 +dr^2 +dz^2)$, $d\tilde{s}^2=d\phi^2$ and the warping function $F(r)=r/(1 +\frac{B_0r^2}{4}).$ 
In the Weyl form an eletrovac or vacuum static, cylindrically symmetric metric can be written as 
\bea\label{LCM}
ds^2=-e^{2\psi}dt^2 +e^{-2\psi}[e^{2 \lambda}(dr^2 +dz^2) +r^2d\phi^2]
\eea
where $\psi$ and $\lambda$ are no-where vanishing smooth functions of $r$ and $z$. We see that for $\psi=2\lambda =ln U_B$, the metric (\ref{MMM}) is obtained from the metric (\ref{LCM}). Thus we can see Melvin type metric in Weyl form as
\bea\label{MTM}
ds^2=e^{2f(r)}(-dt^2 +dr^2 +dz^2) +r^2e^{-2f(r)}d\phi^2.
\eea 
\indent In the literature of differential geometry, the notion of pseudosymmetry generalizes the notions of (local) symmetry ($\nabla R=0$) and semisymmetry ($R\cdot R=0$ ) introduced by Cartan \cite{C26, C46}. Such a notion of pseudosymmetry has a great importance in the study of general relativity and cosmology due to its applications. Precisely, many spacetimes are pseudosymmetric. The extention of this notion to other curvature tensors are called pseudosymmetric type curvature conditions. By curvature restricted geometric structures of a manifold mean the structures arose due to restriction of covariant derivative(s) on various curvature tensors of that manifold. Again, the study of local symmetry has been extended to various concepts such as recurrency \cite{R46, R49,Ru49}, generalized recurrency \cite{SAR13, SK14, SP10, SR10, SR11, SRK17, SRK15}, curvature $2$-forms of recurrency \cite{B87, LR89}, weakly symmetry \cite{TB89}, Chaki pseudosymmetry \cite{C87} etc. our aim in this paper is to investigate such type of geometry by means of restriction on several curvature tensors of Melvin magnetic metric \eqref{MMM}. We mention that Robertson-Walker spacetimes \cite{ADEHM14, DDHKS00, DK99}, the warped products of $1$-dimensional base and $3$-dimensional fibre, are pseudosymmetric. Also Schwarzschild spacetime \cite{GP09}, Robinson-Trautman spacetime \cite{SAA18} etc. are models of pseudosymmetric warped product spacetimes of $2$-dimensional base and $2$-dimensional fibre. The curvature properties of an warped product metric have also been investigated in \cite{SM}  Thus, it is worthy to investigate the geometric structures of a warped product spacetime like Melvin of $3$-dimensional base and $1$-dimensional fibre.  \\
\indent The planning of this paper is as follows: Section $2$ is devoted to the preliminaries of various curvature tensors and curvature restricted geometric structures. In section $3$ we obtain the geometric structures of Melvin type metric \eqref{MTM} and also obtain the conditions for which it is conformally flat, pseudosymmetric and Roter type manifold. The curvature properties of Melvin magnetic spacetime have been determined in section $4$ and it is shown that such metric is pseudosymmetric and has pseudosymmetric Maxwell tensor. Also the tensor $C\cdot R- R\cdot C$ is lineraly dependent with $Q(S,C)$ i.e., a generalized Einstein metric condition is satisfied on this spacetime. Moreover, it is $2$-quasi-Einstein and Roter type. Finally we add some conclusion in section $5$.
%
%%%%%%%%%%%%%%%%%%%%%%%%%%%%%%%%%%%%%%%%%%%%%%%%%%%%%%%%%%%%%%%%%%%%%%%%%%%%%%%%%%%%%%%%%%%%%%%%%%%%%%%%%%%%%%%%%%%%%%%%%%
%                                                               Preliminaries                                           %
%%%%%%%%%%%%%%%%%%%%%%%%%%%%%%%%%%%%%%%%%%%%%%%%%%%%%%%%%%%%%%%%%%%%%%%%%%%%%%%%%%%%%%%%%%%%%%%%%%%%%%%%%%%%%%%%%%%%%%%%%%
%
\section{\bf{Preliminaries}}
%==========================
\indent In this section we will review the notations and definitions of various curvature tensors and some useful curvature restricted geometric structures. For this the endomorphisms $U\wedge_E V$, $\mathcal R$, $\mathcal C$, $\mathcal P$, $\mathcal W$ and $\mathcal K$ on $M$ are respectively defined as \cite{DHJKS14, SK19}
\bea\label{gct}
(U\wedge_E V)Z &=& E(V,Z)U-E(U,Z)V , \nonumber \\
\mathcal{R}(U,V) &=& \left[ \nabla_U, \nabla_V\right]- \nabla_{[U,V]}, \nonumber\\
\mathcal{C} &=& \mathcal{R}-\frac{1}{(n-2)}(\wedge_g\mathcal{J}+\mathcal{J}\wedge_g-\frac{\kappa}{n-1}\wedge_g ),\nonumber \\
\mathcal{P} &=& \mathcal{R}-\frac{1}{(n-1)}\wedge_S ,\nonumber \\
\mathcal{W}&=&\mathcal{R}-\frac{\kappa}{n(n-1)}\wedge_g \nonumber \; \; \mbox{and}\\
\mathcal{K}&=&\mathcal{R}-\frac{1}{(n-2)}(\wedge_g\mathcal{J}+\mathcal{J}\wedge_g) \nonumber
\eea
where $E$ is a symmetric $(0,2)$-tensor, $\mathcal J$ is the Ricci operator defined by $S(X,Y)=g(X,\mathcal J Y)$ and throught this paper we consider $X$, $Y$, $Z$, $Z_i$, $U$, $V$ $\in$ $\chi(M)$, the Lie algebra of all smooth vector fields on $M$. The Kulkarni-Nomizu product $E\wedge F$ for two symmetric $(0,2)$ tensors $E$ and $F$ is defined by \cite{DG02, DGHS11, DGHS98, DH03, G02, SK17}
\bea
(E\wedge F)(X,Y,U,V)&=&E(X,V)F(Y,U)-E(X,U)F(Y,V)   \nonumber\\
&+&E(Y,U)F(X,V)-E(Y,V)F(X,U).\nonumber
\eea
Now we define the $(0,4)$-tensor $H$ corresponding to an endomorphism $\mathcal{H}(Z_1,Z_2)$ on $M$ as
\bea
H(Z_1,Z_2,Z_3,Z_4)=g(\mathcal{H}(Z_1,Z_2)Z_3,Z_4)
\eea
Replacing $\mathcal{H}$ in above by $\mathcal C$ (resp., $\mathcal P$, $\mathcal W$, $\mathcal K$ and $Z_1\wedge_g Z_2$) one can easily obtain the conformal curvature tensor $C$ (resp., the projective curvature tensor $P$, the concircular curvature tensor $W$, the conharmonic curvature tensor $K$ and the Gaussian curvature tensor $G$). Let $(M,g)$ be covered with a system of coordinate charts $(\mathcal U, x^{\alpha})$. Then the local components of the $(0,4)$-tensors $E\wedge F$, $R$, $C$, $P$, $W$ and $K$ are written as
\bea
(E\wedge F)_{abcd} &=& E_{ad}F_{bc}- E_{ac}F_{bd} +E_{bc}F_{ad} -E_{bd}F_{ac},\nonumber \\
R_{abcd} &=& g_{a\alpha}(\partial_{d}\Gamma^{\alpha}_{bc} - \partial_c \Gamma^{\alpha}_{bd} + \Gamma^{\beta}_{bc}\Gamma^{\alpha}_{\beta d} - \Gamma^{\beta}_{bd}\Gamma^{\alpha}_{\beta c}), \nonumber \\
C_{abcd} &=& R_{abcd} - \frac{1}{(n-2)}(g\wedge S)_{abcd} + \frac{\kappa}{(n-1)(n-2)}G_{abcd}, \nonumber \\
P_{abcd} &=& R_{abcd} -\frac{1}{(n-1)}(g_{ad}S_{bc} - g_{bd}S_{ac}), \nonumber \\
W_{abcd} &=& R_{abcd} - \frac{\kappa}{n(n-1)}G_{abcd} \nonumber \; \; \mbox{and} \\
K_{abcd} &=& R_{abcd} - \frac{1}{(n-2)}(g\wedge S)_{abcd} \nonumber
\eea
where $\Gamma^a_{bc}$ are Christoffel symbols of $2$nd kind, $\partial_\alpha=\frac{\partial}{\partial x^\alpha}$ and $G_{abcd}=\frac{1}{2}(g\wedge g)_{abcd}$.
Let $T$ be a $(0,k)$-tensor, $k\geq 1$, and the operation of the endomorphisms $\mathcal{A}(X,Y)$ and $X\wedge_B Y$ on it give rise two $(0,k+2)$-tensors $A\cdot T$ and $Q(B,T)$ \cite{DG02, DGHS98, SDHJK15, SK14, T74} respectively given as
\beb
(A\cdot T)(Z_1,Z_2,\cdots,Z_k;U,V)&=&(\mathcal{A}(U,V)\cdot T)(Z_1,Z_2,\cdots,Z_k)\\
&&\hspace{-1in} =-T(\mathcal{A}(U,V)Z_1,Z_2,\cdots,Z_k)- \cdots -T(Z_1,Z_2,\cdots,\mathcal{A}(U,V)Z_k) \;\;and
\eeb
\beb
Q(B,T)(Z_1,Z_2,\cdots,Z_k;U,V)&=&((U\wedge_B V)\cdot T)(Z_1,Z_2,\cdots,Z_k)\\
&&\hspace{-2in} =-T((U\wedge_B V)Z_1,Z_2,\cdots,Z_k)- \cdots -T(Z_1,Z_2,\cdots,(U\wedge_B V)Z_k)\\
&&\hspace{-2in} = B(U, Z_1) T(V,Z_2,\cdots,Z_k) + \cdots + B(U, Z_k) T(Z_1,Z_2,\cdots,V)\\
&&\hspace{-2in} - B(V, Z_1) T(U,Z_2,\cdots,Z_k) - \cdots - B(V, Z_k) T(Z_1,Z_2,\cdots,U).
\eeb
where $A(U,V,Z_1,Z_2)=g(\mathcal{A}(U,V)Z_1,Z_2)$ and $B$ is a symmetric $(0,2)$ tensor.
The local expressions of the tensors $A\cdot T$ and $Q(B,T)$ are 
\bea
(A\cdot T)_{a_1a_2...a_k\alpha \beta} &=& -g^{rs}[A_{\alpha \beta a_1s}T_{ra_2...a_k} +\cdots +A_{\alpha \beta a_ks}T_{a_1a_2...r}],\nonumber \\
Q(B,T)_{a_1a_2...a_k \alpha \beta} &=& B_{\beta a_1}T_{\alpha a_2...a_k} + \cdots + B_{\beta a_k}T_{a_1a_2...\alpha} \nonumber \\ &-& B_{\alpha a_1}T_{\beta a_2...a_k} - \cdots - B_{\alpha a_k}T_{a_1a_2...\beta}\nonumber.
\eea
\begin{defi}\label{defi2.1}(\cite{AD83, C46, DDP94, DDVV94, D92, SK14, SK18, S82, S84, S85})
If on $(M,g)$ the relation $A\cdot T=0$ holds then it is said to be $T$- semisymmetric due to $A$ and said to be $T$-pseudosymmetric due to $A$ if $A\cdot T=\mathcal L_T Q(g,T)$ holds on $M$, where $\mathcal L_T$ is smooth function on $\left\{x\in M:Q(g,T)_x\ne 0\right\}.$
\end{defi}
If we replace $A$ by $R$ and $T$ by $R$ (resp., $S$, $C$, $W$ and $K$ ) in the above definition then we obtain semisymmetric (resp., Ricci, conformally, concircularly and conharmonically semisymmetric) and pseudosymmetric (resp., Ricci, conformally, concircularly and conharmonically pseudosymmetric) manifolds respectively. Also replacing $A$ by $C$, $W$ and $K$ we obtain the corresponding pseudosymmetric type curvature conditions.
\begin{rem}\label{rem1}
1. In $Theorem \; 4.1$ of \cite{DDP94} it was proved that every warped product $\bar{M}\times_f\tilde{N}$ with $dim \bar{M}=1$ and $\dim \tilde{N}=3$ satisfies 
\bea\label{WT}
C\cdot C=\mathcal L_C'\;Q(g,C)
\eea
for some function $\mathcal L_C'$ and Robertson-Walker spacetime \cite{ADEHM14, DDHKS00, DK99} admits this property.\\
2. In $Theorem \; 2$ of \cite{D91} it was given that every warped product $\bar{M}\times_f\tilde{N}$ with $dim \bar{M}=2$ and $\dim \tilde{N}=2$ satisfies 
\bea\label{DT}
R\cdot R- Q(S,R)=\mathcal L' \; Q(g,C)
\eea
for some function $\mathcal L'$ and Schwarzschild spacetime \cite{GP09}, Robinson-Trautman spacetime \cite{SAA18} are examples of such warped products. 
\end{rem}
The manifold $M$ is called a $k$-quasi-Einstein manifold if the rank of $(S-\alpha g)$ is $k$, $0\leq k\leq (n-1)$, for some scalar $\alpha$. For $k=1$ (resp., $k=0$) the manifold is called quasi-Einstein (resp., Einstein) and the Ricci tensor locally takes the form \cite{S1, SHY09, S2} $$S_{ab}=\alpha g_{ab} + \beta \Pi_a \otimes \Pi_b$$ $$(resp., S_{ab}=\alpha \; g_{ab})$$ where $\alpha$, $\beta$ are scalars and $\Pi_a$ are the components of the covector $\Pi$. Again on the class of non-quasi-Einstein manifolds, the notion of pseudo quasi-Einstein manifolds was introduced in \cite{S09} by decomposing the Ricci tensor $$S_{ab}=\alpha g_{ab} + \beta \Pi_a\otimes \Pi_b + \gamma F_{ab}$$ where $\gamma$ is also a scalar and $F_{ab}$ are components of a trace free $(0,2)$-tensor $F$ such that $F(Y,U)=0$, $U$ is the associated vector field of $\Pi$. In particular, for $F_{ab}=\Pi_a\otimes \delta_b + \delta_a\otimes \Pi_b$, $\delta_a$ are the components of the covector $\delta$, it is generalized quasi-Einstein due to Chaki \cite{C01}. We note that Siklos spacetime \cite{SDD} is quasi-Einstein and Som-Raychaudhuri spacetime \cite{SK16} is $2$-quasi-Einstein as well as generalized quasi-Einstein. For curvature properties of Vaidya metric, we refer the reader to see \cite{SKS17}.
\begin{defi}(\cite{B87, SK19})
If on a semi-Reimannian manifold $M$ the relation
\bea
n_0g + n_1S + n_2S^2 + n_3S^3 + n_4S^4 = 0, \nonumber \; \; n_4\ne 0
\eea
holds where $S^j(Z_1,Z_2)=g(Z_1,\mathcal J^{j-1}Z_2)$ and $n_0$, $n_1$, $n_j$ $\in$ $C^\infty(M)$, the ring of all smooth functions on $M$ $(2\le j \le 4)$ then it is called $Ein(4)$ manifold. For $n_4=0$ (resp., $n_3=n_4=0$) it is $Ein(3)$ (resp., $Ein(2)$) manifold.
\end{defi}
We mention that Siklos spacetime \cite{SDD} is $Ein(2)$ while Lifshitz spacetime \cite{SSC19} is $Ein(3)$.
\begin{defi}\label{defi2.4}(\cite{DGJPZ13, DGJZ16, DGPV15, SDHJK15, SK16, SK19, S15})
	If the tensor $R$ of $M$ can be expressed as  $$R=e_1S\wedge S + e_2S\wedge\ S^2 + e_3S^2\wedge S^2 + e_4g\wedge g + e_5g\wedge S + e_6 g\wedge S^2$$ for some $e_i\in C^\infty(M)$, $1\leq i\leq 6$, then it is called generalized Roter type manifold. For $e_2=e_3=e_6=0$ it is Roter type manifold (\cite{D03, DG02, DGPV11, DK03, G07}).
\end{defi}
\begin{defi}(\cite{DHJKS14, G78, SB})
If on a semi-Riemannian manifold $M$ the Ricci tensor satisfies the condition
$$\sum_{Z_1,Z_2,Z_3}\nabla_{Z_1}S(Z_2,Z_3)=0$$
$$(resp., \; \nabla_{Z_1}S(Z_2,Z_3) -\nabla_{Z_2}S(Z_1,Z_3)=0)$$ where $\sum$ denotes the cyclic sum over $Z_1$, $Z_2$, $Z_3$, then it is said to be of cyclic parallel Ricci tensor (resp., Codazzi type Ricci tensor).
\end{defi}
\begin{defi}(\cite{SK12, SK19, SJ06})
A weakly symmetric manifold is defined by the equation
\beb
\nabla_{\alpha}  R_{abcd} = \Pi_{\alpha} R_{abcd} + \Phi_a R_{\alpha,b,c,d} + \overline \Phi_b R_{a\alpha cd} + \Psi_c R_{ab\alpha d} + \overline \Psi_d R_{abc\alpha}
\eeb
where $\Pi, \Phi, \overline \Phi, \Psi$ and $\overline \Psi$ are covectors on $M$ with local components $\Pi_a, \; \Phi_a, \; \overline \Phi_a, \; \Psi_a, \;\mbox{and} \; \overline \Psi_a$ respectively. In particular, if $\frac{1}{2}\Pi = \Phi = \overline \Phi = \Psi = \overline \Psi$ (resp., $\Phi = \overline \Phi = \Psi = \overline \Psi = 0$), then the manifold reduces to Chaki pseudosymmetric manifold \cite{C87} (resp., recurrent manifold \cite{R46, R49, W50}).
\end{defi}
\begin{defi}(\cite{MM12, MM1210})
Let $H$ be a $(0,4)$-tensor and $E$ be a symmetric $(0,2)$-tensor corresponding to the endomorphism $\mathcal E$ on $M$. Then the tensor $A$ on $M$ is said to be $H$-compatible if $$\sum_{Z_1,Z_2,Z_3}H(\mathcal E Z_1,X,Z_2,Z_3) = 0$$ holds. Again an $1$-form $\Phi$ is said to be $H$-compatible if $\Phi\otimes\Phi$ is $H$-compatible.
\end{defi}
Replacing $H$ by $R$, $W$, $K$, $C$ and $P$ we can define, respectively, Reimann compatibility, concircular compatibility, conharmonic compatibility, conformal compatibility and projective compatibility.
\begin{defi}\label{2forms}
Let $H$ be a $(0,4)$-tensor and $E$ be a symmetric $(0,2)$-tensor corresponding to the endomorphisms $\mathcal E$ on $M$. Then the corresponding curvature $2$-forms $\Omega^m_{(H)l}$ (\cite{B87, LR89}) are recurrent if (\cite{MS12, MS13, MS14})
\beb
\sum_{Z_1,Z_2,Z_3}(\nabla_{Z_1}H)(Z_2,Z_3,X,Y) = \sum_{Z_1,Z_2,Z_3}\Pi(Z_1)H(Z_2,Z_3,X,Y)
\eeb
and the $1$-forms $\Lambda_{(E)l}$ (\cite{SKP03}) are recurrent if
$$(\nabla_{Z_1}E)(Z_2,X)-(\nabla_{Z_2}E)(Z_1,X)=\Pi(Z_1)E(Z_2,X)-\Pi(Z_2)E(Z_1,X)$$ for some covector $\Pi$.
\end{defi}
\begin{defi}(\cite{P95, V85})
Let $L(M)$ be the vector space formed by all covectors $\theta$ on $M$ satisfying $$\sum_{Z_1,Z_2,Z_3}\theta(Z_1)H(Z_2,Z_3,X,Y)=0$$ where $H$ is a $(0,4)$ tensor. Then $M$ is said to be a $H$-space by Venzi if $dimL(M)\geq 1$.
\end{defi}
%
%%%%%%%%%%%%%%%%%%%%%%%%%%%%%%%%%%%%%%%%%%%%%%%%%%%%%%%%%%%%%%%%%%%%%%%%%%%%%%%%%%%%%%%%%%%%%%%%%%%%%%%%%%%%%%%%%%%%%%%%%%%%
%                                          Curvature Restricted Geometric Structure                                        %
%%%%%%%%%%%%%%%%%%%%%%%%%%%%%%%%%%%%%%%%%%%%%%%%%%%%%%%%%%%%%%%%%%%%%%%%%%%%%%%%%%%%%%%%%%%%%%%%%%%%%%%%%%%%%%%%%%%%%%%%%%%%
%
\section{\bf{Geometric structures of Melvin type metric}}
%=======================================================
%
Let $\bar{M}=\left\lbrace \; (t,r,z)\; : \; r > 0 \; \right\rbrace $ be an open connected non-empty subset of $\mathbb{R}^3$ and on $\bar{M}$ we define the metric $-\bar{g}_{11}=\bar{g}_{22}=\bar{g}_{33}=e^{2f(r)}$. Then the Melvin type spacetime in Weyl form is a warped product $M=\bar{M}\times_F \tilde{N}$ of $3$-dimensional base $(\bar{M}, \bar{g})$ and $1$-dimensional fibre $(\tilde{N},\tilde{g})=(S^1(1), d\phi^2)$ with the warping function $F(r)=re^{-f(r)}$.\\ 
The non-zero values of the local components $\bar{R}_{abcd}$,  $\bar{S}_{ab}$, $\bar{K}_{abcd}$ and $\bar{\kappa}$ (upto symmetry) are
%=============================================================================================%
%                                 \bar{R}, \bar{S}, \bar{\kappa}, \bar{K},                    %
%=============================================================================================%
\be\label{bs}
\left\{\begin{array}{c}
\bar{R}_{1212}=e^{2f}f''=-\bar{R}_{2323}, \; \bar{R}_{1313}=e^{2f}f'^2; \; \bar{S}_{11}=-(f'^2+f'')=-\bar{S}_{33},\\  \bar{S}_{22}=2f''; \; \bar{\kappa}=2e^{-2f}(f'^2+2f'');\\ \bar{K}_{1212}=-e^{2f}(f'^2+2f'')=\bar{K}_{1313}=-\bar{K}_{2323}.
\end{array}\right.
\ee
For every $3$-dimensional manifold the conformal curvature tensor vanishes and hence  $\bar{C}=0$ of $(\bar{M},\bar{g})$. The values of $\bar{K}_{abcd,\; \alpha}$, $(\bar{R}\cdot \bar{R})_{abcd\alpha \beta}$ and $Q(\bar{g},\bar{R})_{abcd\alpha \beta}$ (upto symmetry) are
%============================================================================================%
%                   \nabla\bar{K}, \bar{R}\cdot \bar{R}, Q{\bar{g},\bar{R}}                  %
%============================================================================================%
\be\label{bbs}
\left\{\begin{array}{c}
\bar{K}_{1212,2}=2e^{2f}(f'^3+f'f''-f''')=\bar{K}_{1313,2}=-\bar{K}_{2323,2};\\ 
(\bar{R}\cdot \bar{R})_{132312}=-e^{2f}f''(f'^2-f'')=-(\bar{R}\cdot \bar{R})_{121323},\\ Q(\bar{g},\bar{R})_{132312}=-e^{4f}(f'^2-f'')=-Q(\bar{g},\bar{R})_{121323}.
\end{array}\right.
\ee
From \eqref{bs} and \eqref{bbs} we get the following geometric properties for $(\bar{M},\bar{g})$:
\begin{thm}\label{bpr}
The 3-dimensional base $(\bar{M},\bar{g})$ admits the following geometric structures:\\
1.  $\bar{R}\cdot \bar{R}=\mathcal L_{\bar R}\;Q(\bar{g},\bar{R})$ for $\mathcal L_{\bar R}=e^{-2f}f''$ i.e., pseodosymmetric,\\
2. Rank of $(\bar{S}-\alpha \bar{g})=1$ for $\alpha = e^{-2f}(f'^2+f'')$ i.e., it is quasi-Einstein,\\
3. satisfies the $Ein(2)$ condition $S^2+\mu_1S+ \mu_2g=0$ for $$\mu_1=-e^{-2f}(f'^2+3ff''),\;\;  \mu_2=2e^{-4f}f''(f'^2+f''),$$
4. if $(f'^2+2f'')\ne 0$ then the conharmonic tensor is recurrent with $1$-form of recurrency $$\Pi_1=0, \; \;  \Pi_2=\frac{2(f'''-f'f''-f'^3)}{(f'^2+2f'')}, \; \;  \Pi_3=0$$
5. if $(f'^2+f'') \ne 0$ then the Ricci $1$-form is recurrent with $1$-form of recurrency $$\Pi_1=0, \; \; \Pi_2= -\frac{(f'''-f'f''-f'^3)}{(f'^2+f'')}, \; \;  \Pi_3=0.$$
\end{thm}
The non-zero values (upto symmetry) of $R_{abcd}$, $S_{ab}$, $\kappa$ and $C_{abcd}$ of the metric \eqref{MTM} are as follows:
%===================================================================%
%------------------------- R,     S,    \kappa, C ----------------- %
%===================================================================%
\be
\begin{array}{c}
R_{1313}=e^{2f}f'^2,\; \; R_{1212}=e^{2f}f''=-R_{2323}, \; \; R_{1414}=e^{-2f}rf'(1-rf')=-R_{3434},\\ R_{2424}=e^{-2f}r(3f'-2rf'^2+rf''); \; \; S_{11}=-\frac{f'+rf''}{r}=-S_{33}, \\ S_{22}=-\frac{3f'-2rf'^2-rf''}{r},\; \;  S_{44}=e^{-4f}r(f'+rf''); \; \; \kappa=\frac{2e^{-2f}}{r}(rf''+rf'^2-f'); \\ C_{1212}=\frac{e^{2f}}{3r}(rf''-2rf'^2+2f')=-2C_{1313}=\frac{e^{-4f}}{r^2}C_{1414}=-C_{2323}=\frac{2e^{-4f}}{r^2}C_{2424}=-\frac{e^{-4f}}{r^2}C_{3434}.
\end{array}\nonumber
\ee
\begin{lem}\label{CFlem}
The Weyl conformal tensor of Melvin type metric \eqref{MTM} vanishes if and only if $(rf''-2rf'^2+2f')=0$.
\end{lem}
\begin{proof}
From the values of the non-zero components $C_{ijkl}$ the proof is obvious.
\end{proof}
\begin{ex}
	We see that $f(r)=c_1 + \frac{1}{2}ln\left( \frac{r}{c_1r+2}\right) $ is the general solution of $rf''-2rf'^2 +f'=0$, where $c_1$ and $c_2$ are arbitrary constants. In particular, we take $c_1=0$ and $c_2=1$ such that $f(r)=\frac{1}{2}ln\frac{r}{r+2}$. Then the metric \eqref{MTM} with such $f(r)$ takes the form $$ds^2=\frac{r}{(r+2)}\left( -dt^2+dr^2+dz^2\right) + r(r+2)d\phi^2 $$ of which conformal curvature tensor vanishes.
\end{ex}
From above $Lemma$ we get 
$$U_C=\left\lbrace \; x\in M \; : \; (C)_x \ne 0 \;\right\rbrace=\left\lbrace \; (t,r,z,\phi) \; : \; (rf''-2rf'^2+2f') \ne 0 \; \right\rbrace.$$
The non-zero values (upto symmetry) of $(R\cdot R)_{abcd\alpha \beta}$ and $Q(g,R)_{abcd\alpha \beta}$ are given by
%===================================================================%
%------------------------R.R  ,   Q(g,R) ---------------------------%
%===================================================================%
\be
\begin{array}{c}
(R\cdot R)_{132312}=-e^{2f}f''(f'^2-f'')=-(R\cdot R)_{121323}, \\ (R\cdot R)_{142412}=-e^{-2f}rf''(4f'-3rf'^2+rf'')=(R\cdot R)_{243423},\\ (R\cdot R)_{133414}=-e^{-2f}f'^2(1-rf')(1-2rf')=-(R\cdot R)_{131434}, \\ (R\cdot R)_{122414}=e^{-2f}f'(1-rf')(3f'-2rf'^2+2rf'') =(R\cdot R)_{232434},\\ (R\cdot R)_{121424}=-e^{-2f}(3f'-2rf'^2+rf'')(f'-rf'^2-rf'')=(R\cdot R)_{233424}; \\ Q(g,R)_{132312}=-e^{4f}(f'^2-f'')=-Q(g,R)_{121323},\; Q(g,R)_{142412}=-r(4f'-3rf'^2+rf'') \\  =Q(g,R)_{243423},\; Q(g,R)_{133414}=-rf'(1-2rf')=Q(g,R)_{131434}, Q(g,R)_{122414} \\ =r(3f'-2rf'^2+2rf'')=Q(g,R)_{232434}, \; Q(g,R)_{121424}=r(f'-rf'^2-rf'')=Q(g,R)_{233424}.
\end{array}\nonumber
\ee
From the above components we see that metric \eqref{MTM} is not pseudosymmetric but it is pseudosymmetric under certain condition.
\begin{lem}\label{PSlem}
The metric \eqref{MTM} satisfies the condition $R\cdot R=\mathcal L_R\;Q(g,R)$ with $\mathcal L_R=\frac{e^{-2f}}{r}(f'-rf'^2)$ if \;  $(rf''+rf'^2-f')=0$.
\end{lem}
\begin{proof}
We consider the endomorphism $\mathcal{D}(X,Y)=\mathcal{R}(X,Y)-\mathcal L_R'X\wedge_gY$, $\mathcal L_R'\in C^{\infty}(M)$. Then by operating this endomorphism on the tensor $R$, we obtain the $(0,6)$-tensor $(D\cdot R)(U_1,U_2,U_3,U_4,X,Y)$ as 
\bea
(D\cdot R)(U_1,U_2,U_3,U_4,X,Y)&=&(\mathcal{D}(X,Y)\cdot R)(U_1,U_2,U_3,U_4) \nonumber\\
&=& (R\cdot R)(U_1,U_2,U_3,U_4,X,Y)-\mathcal L_R'\; Q(g,R)(U_1,U_2,U_3,U_4,X,Y).\nonumber
\eea
Using the values of the components $R\cdot R_{ijklmn}$, $Q(g,R)_{ijklmn}$ and putting $\mathcal L_R'=\mathcal L_R$ in above, we see that the non-zero components (upto symmetry) of the tensor  $D\cdot R$ are $(D\cdot R)_{132312}$, $(D\cdot R)_{142412}$, $(D\cdot R)_{121323}$, $(D\cdot R)_{243423}$, $(D\cdot R)_{121424}$ and $(D\cdot R)_{233424}$.
Now, if  $(rf''+rf'^2-f')=0$ then $(D\cdot R)_{ijklmn}=0$ for $1\le i,j,k,l,m,n \le 4$.
Hence $R\cdot R=\mathcal L_R\;Q(g,R)$. This completes the proof.
\end{proof}
\begin{ex}
	It is easy to see that $f(r)=c_1 + ln(r^2+2c_2)$ is the general solution of $rf''+rf'-f'^2=0$, where $c_1$ and $c_2$ are arbitrary constants. In particular, we take $c_1=0$ and $c_2=\frac{1}{2}$ such that $f(r)=ln(1+r^2)$. Then the metric \eqref{MTM} with such $f(r)$ takes the form $$ds^2=\frac{1}{(1+r^2)^2}\left(-dt^2 +dr^2 + dz^2 \right) + r^2(1+r^2)^2 d\phi^2$$ which is pseudosymmetric.
	\end{ex}

Again the non-zero components $C_{abcd,\; \alpha}$, $(C\cdot C)_{abcd\alpha \beta}$, $Q(g,C)_{abcd\alpha \beta}$ and $Q(S,R)_{abcd\alpha \beta}$ (upto symmetry) are given by
%=========================================================================%
%                  \nabla C, C.C, Q(g,C), Q(S,R)                          %
%=========================================================================%
\be
\begin{array}{c}
C_{1313,2}=\frac{2e^{2f}}{3r^2}(2rf'(2f'-2rf'^2+3rf'')+(2f'-2rf''-r^2f'''))=-2C_{1212,2} \\ =-\frac{2e^{4f}}{r^2}C_{1414,2}=2C_{2323,2}=-2e^{4f}C_{2424,3}=-e^{4f}C_{3434,2},\\ C_{1213,3}=\frac{e^{2f}}{r}f'(2f'-2rf'^2+rf'')=C_{1323,1}=-\frac{e^{4f}}{r^2}C_{1424,1}=\frac{e^{4f}}{r^2}C_{2434,3}; \\ r^2e^{-4f}(C\cdot C)_{132312}=\frac{e^{-2f}}{3}(2f'-2rf'^2+rf'')^2=-(C\cdot C)_{142412}=(C\cdot C)_{122414} \\ =-(C\cdot C)_{133414}=-r^2e^{-4f}(C\cdot C)_{121323}=-(C\cdot C)_{243423}=(C\cdot C)_{131434}=(C\cdot C)_{232434}; \\ r^2e^{-4f}Q(g,C)_{132312}=-Q(g,C)_{142412}=r(2f'-2rf'^2+rf'')=-r^2e^{-4f}Q(g,C)_{121323} \\ =Q(g,C)_{133414}=-Q(g,C)_{122414}=-Q(g,C)_{131434}=Q(g,C)_{243423}=-Q(g,C)_{232434};\\ Q(S,R)_{132312}=\frac{e^{2f}}{r}(f'^3(3-2rf')+f''(f'-rf'+rf''))=-Q(S,R)_{121323}, \\ Q(S,R)_{142412}=-e^{-2f}r(f'^3(3-2rf')+f''(5f'-3rf'^2+rf''))=Q(S,R)_{243423}, \\ Q(S,R)_{133414}=e^{-2f}f'(f'+rf'')=-Q(S,R)_{131434}, \\ Q(S,R)_{122414}=e^{-2f}f'(3-2rf')(f'+rf'')=Q(S,R)_{232434}, \\ Q(S,R)_{121424}=-e^{-2f}(3f'-2rf'^2+rf'')(f'-rf'^2-rf'')=Q(S,R)_{233424}.
\end{array}\nonumber
\ee
The values of the above components help us to obtain the following relations on $U_C$:
\bea\label{WPS}
C\cdot C=\mathcal L_C\; Q(g,C) \; \; for \; \;  \mathcal L_C=\frac{e^{-2f}}{3r}(2f'-2rf'^2+rf''),
\eea
\bea\label{SPS}
R\cdot R-Q(S,R)=\mathcal L\;Q(g,C) \; \; for \; \; \mathcal L=-\frac{e^{-4f}f'}{3r\mathcal L_C}(3f'^2-2rf'^3+f'').
\eea
The relation \eqref{WPS} shows that the metric \eqref{MTM} has pseudosymmetric Weyl tensor. Throughout this paper we consider the smooth functions $\mathcal L_R$, $\mathcal L_C$ and $\mathcal L$ as defined respectively in $Lemma$ \ref{PSlem}, relation \eqref{WPS} and \eqref{SPS}. By a straightforward calculation the local components of the tensors $g\wedge g$, $g\wedge S$, $g\wedge S^2$, $S\wedge S$ and $S\wedge S^2$ can easily be computed. Now these tensors help us to decompose the tensor $R$ as
\bea\label{GRT}
R=\mathcal L_{12}\;g\wedge S+ \mathcal L_{13}\;g\wedge S^2+ \mathcal L_{22}\;S\wedge S+ \mathcal L_{23}\;S\wedge S^2
\eea
where $\mathcal L_{12}=\frac{r(e^{2f}\mathcal L_R+f'')}{2(re^{2f}\mathcal L_R+f')}+\frac{f'}{2\mathcal L_r}$, $\mathcal L_{13}=-\frac{e^{4f}r^2f''}{4(re^{2f}\mathcal L_R+f')\mathcal L_r^2}$, $\mathcal L_{22}=(\mathcal L_{13}\mathcal L_re^{-2f}-\frac{r^2e^{4f}}{2\mathcal L_r^2}\mathcal L_R)$, $\mathcal L_{23}=2re^{2f}\mathcal L_{13}\mathcal L_R$ and $\mathcal L_r=(f'+rf'')$. Now contracting the relation \eqref{GRT} we get 
\bea\label{Ein3}
a_{11}\; g+ a_{22}\; S+ a_{33}\; S^2+ a_{44}\;S^3=0
\eea
where $a_{11}=-(2\mathcal L_R+ \frac{e^{-2f}}{r}\mathcal L_ra_{22})$, $a_{22}=-\frac{re^{2f}\mathcal L_R}{re^{2f}\mathcal L_R+ f'}$, $a_{33}=-\frac{re^{2f}}{\mathcal L_r}(a_{22}+ \frac{2re^{2f}}{\mathcal L_r}\mathcal L_R)$, $a_{44}=-\frac{r^2}{\mathcal L_r^2}e^{4f}a_{22}$ and $\mathcal L_r$ is defined above. Further from the non-zero components $C_{abcd}$ and $C_{abcd,\;\alpha}$, on $U_C$ we obtain the curvature $2$-forms of recurrency for conformal curvature tensor with $1$-form of recurrency as
\bea\label{c2f}
\Pi_1=0, \; \;  \Pi_2=-\frac{3re^{2f}}{\mathcal L_C}(2f' -2rf'^2 +2r^2f'^3 -2rf'' +3r^2f'f'' -r^2f'''),\; \; \Pi_3= 0,\;\; \Pi_4=0. \nonumber
\eea
From relation \eqref{WPS}--\eqref{Ein3} we can state the following:
\begin{thm}
	The metric \eqref{MTM} admits the following geometric structures on $U_C$:\\
	(i) it has pseudosymmetric Weyl conformal tensor,\\
	(ii) it satisfies the pseudosymmetric type condition $$R\cdot R -Q(S,R)=-e^{-2f}f'\frac{(3f'^2-2rf'^3+f'')}{(2f'-2rf'^2+rf'')}\;Q(g,C),$$
	(iii) it is generalized Roter type manifold,\\
	(iv) it is $Ein(3)$ manifold,\\
	(v) its conformal $2$-forms are recurrent.
	
\end{thm}
\begin{lem}\label{RTlem}
The metric \eqref{MTM} satisfies the Roter type condition on $U_C$ if $rf''+rf'-f'^2=0$.
	\end{lem}
\begin{proof}
$Remark$ $2.2$ of \cite{DPS13} states that a manifold $(M,g)$, $n\ge 4$, is Roter type on $\left\lbrace \; x\in M \; : \; (C)_x \ne 0 \;\right\rbrace $ if it is pseudosymmetric i.e., $R\cdot R=\mathcal L_R' \; Q(g,R)$ and satisfies the pseudosymmetric type condition $R\cdot R -Q(S,R)=\mathcal L' \;  Q(g,C)$, where $\mathcal L_R'$ and $\mathcal L'$ are some functions on $\left\lbrace \; x\in M \; : \; (C)_x \ne 0 \;\right\rbrace$. Hence, $Lemma$\; \ref{PSlem} and the relation \eqref{WPS} imply that the metric \eqref{MTM} satisfies the Roter type condition on $U_C$ if \; $rf''+rf'-f'^2=0$. 
\end{proof}
\begin{rem}\label{Roe}
	If $rf''+rf'-f'^2=0$ then by $Lemma$ \ref{RTlem} and from $Remark$ $2.2$ of \cite{DPS13} we can write the tensor $R$ on $U_C$ for the metric \eqref{MTM} as 
	\bea\label{RT}
	R=\mu\; S\wedge S -2\mu(\mathcal L_R-\mathcal L)\; g\wedge S +(\mu (\mathcal L_R-\mathcal L)^2 -\frac{\mathcal L}{4})\;g\wedge g
	\eea
	where $\mu$ is some function on $U_C$ and is given by $\mu =-\frac{1}{4(\mathcal L_R-\mathcal L)}$.
	\end{rem}
.
\begin{cor}\label{MTMcor}
Thus from $Theorem$ $6.7$ of \cite{DGHS11}  and by the $Lemma$ $\ref{RTlem}$, we obtain the condition $rf''+rf'^2 -f'=0$ as the sufficient condition of the metric \eqref{MTM} to realise the following geometric structures:\\
1. $R\cdot R=\mathcal L_RQ(g,R), \; \; \mathcal L_R=\frac{e^{-2f}}{r}(f'-rf'^2)$,\\
2. $C\cdot R=\mathcal L_RQ(g,R)$,\\
3. $S^2 +\lambda g=0,\;\; \lambda =-3\mathcal L_R(\mathcal L_R-\mathcal L)$,\\
4. $Q(S,C)=C\cdot R -R\cdot C$. 
\end{cor}
%%%%%%%%%%%%%%%%%%%%%%%%%%%%%%%%%%%%%%%%%%%%%%%%%%%%%%%%%%%%%%%%%%%%%%%%%%%%%%%%%%%%%%%%%%%%
%=========================================================================
\section{\bf{Melvin magnetic metric admitting geometric structures}}
%==========================================================================
%
It is easy to see that $f(r)=ln(1+\frac{B_0r^2}{4})$ satisfies the equation $rf''+rf'^2-f'=0$. Hence from $Lemma$ \ref{PSlem}, $Lemma$ \ref{RTlem} and $Corollary$ \ref{MTMcor} we can state the following:
\begin{thm}
The Melvin magnetic metric \eqref{MMM} admits the following properties:\\
(i) scalar curvature $\kappa=0$ and hence $C=K$ and $R=W$,\\
(ii) it is pseudosymmetric and pseudosymmetric due to conformal curvature tensor i.e., satisfies the relations 
$$R\cdot R=\mathcal L_1\;Q(g,R),\; \; C\cdot R=\mathcal L_1\;Q(g,R) \; \; where\;\; \mathcal L_1=\frac{32B_0^2(4-B_0^2r^2)}{(4+B_0^2r^2)^4},$$
(iii) it also satisfies the following pseudosymmetric type curvature conditions
\bea
(a)\;\;&& R\cdot R -Q(S,R)=\mathcal L_2\;Q(g,C),\; \; \mathcal L_2=-\frac{32B_0^2(16+24B_0^2r^2-3B_0^4r^4)}{3(4-B_0^2r^2)(4+B_0^2r^2)^4} \nonumber \\
(b)\;\;&& Q(S,C)=C\cdot R -R\cdot C,\nonumber\\
(c)\;\;&& C\cdot R -R\cdot C=\mathcal L_3\;Q(g,R) +\mathcal L_4\;Q(S,R),\;\; \mathcal L_3=-\frac{2048B_0^2(4-B_0^2r^2)}{\omega (4+B_0^2r^2)^4}, \nonumber \\
&& \mathcal L_4=1+\frac{64}{\omega},\;\; \omega=-16-24B_0^2r^2+3B_0^4r^4,\nonumber
\eea
(iv) curvature $2$-forms for $C$ or $K$ are recurrent with associated $1$-form of recurrency $$\Pi=\left\lbrace 0, 0, -\frac{16B_0^2r}{(4-B_0^2r^2)(4+B_0^2r^2)}, 0 \right\rbrace, $$
(v) it is Roter type spacetime i.e., satisfies $R=N_1\;S\wedge S +N_2\; g\wedge S +N_3\; g\wedge g$\\ where $$N_1=-\frac{3(4-B_0^2r^2)(4+B_0^2r^2)^4}{8192B_0^2},\;\; N_2=\frac{1}{2},\;\; N_3=-\frac{8B_0^2(4-B_0^2r^2)}{(4+B_0^2r^2)^4}$$
(vi) $2$-quasi Einstein for $\alpha = \frac{256B_0^2}{(4+B_0^2r^2)^4}$,\\
(vii) generalized quasi Einstein due to Chaki for $$\alpha=-\frac{256B_0^2}{(4+B_0^2r^2)},\; \; \beta=-1, \; \; \gamma=1, \; \; \Pi=\left\lbrace -\frac{4 \sqrt{2}B_0}{(4+B_0^2r^2)},\; \; 0,\; \; \frac{4\sqrt{2}B_0}{\sqrt{(16+8B_0^2r^2+B_0^4r^4)}},\;\; 0 \right\rbrace ,$$ $$ \delta=\left\lbrace 0, \; \; 0,\;\; \frac{4\sqrt{2}B_0}{(4+B_0^2r^2)},\;\; 0 \right\rbrace, \; \; with\;\; \left\| \Pi\right\|=0 \;\; and \;\; \left\| \delta\right\|=\frac{512B_)^2}{(4+B_0^2r^2)^4},$$
(viii) $Ein(2)$ spacetime, since $S^2+ \lambda\; g=0$ for $\lambda=-\frac{65536B_0^2}{(4+B_0^2r^2)^8},$\\
(ix) Ricci tensor is Reimann compatible only.
\end{thm}
\begin{rem}
It is worthy to note the following facts about the Melvin magnetic spacetime:\\
1. Melvin magnetic spacetime is a pseudosymmetric warped product of $3$-dimensional pseudosymmetric base and $1$-dimensional fibre with warping function $F(r)=r/U_B$, where $U_B=1 +\frac{B_0r^2}{4}$.\\
2. Melvin magnetic spacetime, a warped product of $3$-dimensional conharmonically recurrent base and $1$-dimensional fibre, is not conharmonically recurrent but its curvature $2$-forms for conharmonic tensor are recurrent.\\
3. Melvin magnetic spacetime is a $2$-quasi-Einstein warped product spacetime of $3$-dimensional quasi-Einstein base and $1$-dimensional fibre.\\
4. Referred to the $Remark$ \ref{rem1}, the Melvin magnetic spacetime is a warped product spacetime of $3$-dimensional base and $1$-dimensional fibre satisfying both the relations \eqref{WT} and \eqref{DT}.
\end{rem}
The non zero components of Maxwell tensor from \eqref{MF} are given by
$F_{24}=\frac{8B_0r}{(4+B_0^2r^2)^2}=-F_{42}$. The non-zero components of the tensors $R\cdot F$ and $Q(g,F)$ are $$(R\cdot F)_{1412}=-(R\cdot F)_{1214}=-\frac{16B_0^3r(4-B_0^2r^2)}{(4+B_0^2r^2)}=R\cdot F_{3423}=-R\cdot F_{2334};$$ $$Q(g,F)_{1214}=-Q(g,F)_{1412}=\frac{B_0r}{2}=Q(g,F)_{2334}=-Q(g,F)_{3423}.$$
From above we have the following:
\begin{thm}
$R\cdot F=\mathcal L_F\; Q(g,F)$ where $\mathcal L_F=\frac{32B_0^2(4-B_0^2r^2)}{(4+B_0^2r^2)^4}$ i.e., Maxwell tensor of Melvin spacetime is pseudosymmetric type.
\end{thm}
%%%%%%%%%%%%%%%%%%%%%%%%%%%%%%%%%%%%%%%%%%%%%%%%%%%%%%%%%%%%%%%%%%%%%%%%%%%%%%%%%%%%%%%%%%%%%%%%%%%%%%%%%%%%%%%%%%%%%%%%%%%%%%%%%%%%%%%%%%%%%%
%                                                         Conclusion
%%%%%%%%%%%%%%%%%%%%%%%%%%%%%%%%%%%%%%%%%%%%%%%%%%%%%%%%%%%%%%%%%%%%%%%%%%%%%%%%%%%%%%%%%%%%%%%%%%%%%%%%%%%%%%%%%%%%%%%%%%%%%%%%%%%%%%%%%%%%%%
%
\section{\bf{Conclusion}}
We have studied the curvature restricted geometric properties of Melvin magnetic spacetime, a warped product with $1$-dimensional fibre, and found that it is non-semisymmetric pseudosymmetric warped product with $3$-dimensional pseudosymmetric base. Also it is non-quasi-Einstein $2$-quasi-Einstein warped product spacetime with $3$-dimensional quasi-Einstein base. Moreover its Maxwell tensor is pseudosymmetric type. Hence, Melvin magnetic spacetime is evidently a model of $4$-dimensional pseudosymmetric and $2$-quasi-Einstein warped product manifolds with $1$-dimensional fibre.\\
\textbf{Acknowledgement.} The authors would like to express their gratitude to Deanship of Scientific Research at King Khalid University, Abha, Saudi Arabia for providing funding research groups under the research grant number R. G. P.$2/57/40$. All the algebraic computations are performed by a program in Wolfram Mathematica.

%%%%%%%%%%%%%%%%%%%%%%%%%%%%%%%%%%%%%%%%%%%%%%%%%%%%%%%%%%%%%%%%%%%%%%

\end{document}